\documentclass{amsart}

\usepackage{amssymb}
\usepackage{rotating}
\newcommand{\Z}{{\mathbb{Z}}}

\newcommand{\R}{{\mathbb{R}}}
\newcommand{\C}{{\mathbb{C}}}
\newcommand{\cD}{{\mathcal{D}}}

\newcommand{\cG}{{\mathcal{G}}}
\newcommand{\cH}{{\mathcal{H}}}
\newcommand{\ulc}{{\mathfrak{c}}}
\newcommand{\fs}{{\mathfrak{s}}}
\newcommand{\ft}{{\mathfrak{t}}}
\newcommand{\fu}{{\mathfrak{u}}}
\newcommand{\fC}{{\mathfrak{C}}}

\newcommand{\Irr}{{\operatorname{Irr}}}
\newcommand{\Hom}{{\operatorname{Hom}}}
\newcommand{\End}{{\operatorname{End}}}
\renewcommand{\leq}{\leqslant}
\renewcommand{\geq}{\geqslant}

\newtheorem{thm}{Theorem}[section]
\newtheorem{lem}[thm]{Lemma}
\newtheorem{cor}[thm]{Corollary}
\newtheorem{prop}[thm]{Proposition}
\newtheorem{conj}[thm]{Conjecture}

\theoremstyle{definition}
\newtheorem{exmp}[thm]{Example}
\newtheorem{defn}[thm]{Definition}

\theoremstyle{remark}
\newtheorem{rem}[thm]{Remark}

\begin{document}

\title{Kazhdan--Lusztig cells and the Frobenius--Schur indicator}

\author{Meinolf Geck}
\address{Institute of Mathematics, University of Aberdeen, Aberdeen AB24
3UE, UK}
\email{m.geck@abdn.ac.uk}

\date{October 2011}

\begin{abstract} Let $W$ be a finite Coxeter group. It is well-known
that the number of involutions in $W$ is equal to the sum of the
degrees of the irreducible characters of $W$. Following a suggestion
of Lusztig, we show that this equality is compatible with the 
decomposition of $W$ into Kazhdan--Lusztig cells. The proof uses
a generalisation of the Frobenius--Schur indicator to symmetric
algebras, which may be of independent interest.
\end{abstract}

\maketitle

\section{Introduction} \label{sec0}

Let $G$ be a finite group and assume that all complex irreducible
characters of $G$ can be realised over the real numbers. Then, by 
a well-known result due to Frobenius and Schur, the number of 
involutions in $G$ (that is, elements $g \in G$ such that $g^2=1$) is 
equal to the sum of the degrees of the irreducible characters of $G$.

In this note, we consider the case where $G=W$ is a finite Coxeter group.
Following a suggestion of Lusztig, we show that the above equality is
compatible with the decomposition of $W$ into cells, as defined by 
Kazhdan and Lusztig \cite{KaLu} (in the equal parameter case) and by
Lusztig \cite{Lusztig83} (in general). The proof relies on two basic 
ingredients. The first consists of establishing a suitable generalisation 
of the ``Frobenius--Schur indicator'' to symmetric algebras. This will be 
done in Section~\ref{sec1}, and may be of independent interest. The 
second ingredient is the theory around Lusztig's ring $J$ (originally 
introduced in \cite{Lu2}) or, rather, its more elementary version 
constructed in \cite{my02}; see Section~\ref{sec2}. 

To state the main result, let us fix some notation. Let $S$ be a set 
of simple reflections in $W$. Let $\{c_s \mid s \in S\} \subseteq 
\Z_{\geq 0}$ be a set of ``weights'' where $c_s=c_{s'}$ whenever $s,
s' \in S$ are conjugate in $W$. This gives rise to a weight function
$L \colon W \rightarrow \Z$ in the sense of Lusztig \cite{Lusztig03}; 
for $w \in W$, we have $L(w)=c_{s_1}+\ldots + c_{s_k}$ where $w=
s_1\cdots s_k$ ($s_i \in S$) is a reduced expresssion for~$w$. (The 
original setup in \cite{KaLu} corresponds to the case where $c_s=1$
for all $s \in S$.)  Using the Kazhdan--Lusztig basis of the generic
Iwahori--Hecke algebra associated with $W,L$, one can define partitions 
of $W$ into left, right and two-sided cells. For any such left cell 
$\Gamma$ of $W$, we have a corresponding left $W$-module $[\Gamma]_1$ 
with a standard basis indexed by the elements of $\Gamma$; see \cite{KaLu} 
(equal parameter case) or \cite{Lusztig83} (in general).
  
\begin{thm} \label{propleft} The number of involutions in a left cell 
$\Gamma$ is equal to the number of terms in a decomposition of $[\Gamma]_1$ 
as a direct sum of simple $W$-modules.
\end{thm}

For $W$ of classical type and the equal parameter case, the above result 
(in a somewhat more precise form, see Example~\ref{exp1} below) was first 
obtained by Lusztig \cite[12.17]{LuBook}, using the representation theory 
of a finite reductive group with Weyl group $W$. Our proof works uniformly 
for all $W,L$ (including $W$ of non-crystallographic type). In 
Corollary~\ref{proptwo}, we also obtain a similar result for two-sided 
cells.  Along the way, we establish some properties of left cell modules
which previously were only known to hold in the equal parameter case;
see Corollaries~\ref{cor32} and \ref{cor30}.

\section{Symmetric algebras and the Frobenius--Schur indicator} \label{sec1}
Let $K$ be a field of characteristic $0$ and $\cH$ be a finite-dimensional 
associative $K$-algebra (with $1$). We assume that $\cH$ is split semisimple
and symmetric, with trace form $\tau \colon \cH \rightarrow K$. 
Let $\Irr(\cH)$ be the set of simple $\cH$-modules (up to isomorphism).
For $E \in \Irr(\cH)$, let $\chi_E \colon \cH \rightarrow K$ be the
corresponding character, $\chi_E(h)=\mbox{trace}(h,E)$ for all $h \in \cH$.
We have 
\[ \tau=\sum_{E \in \Irr(\cH)} c_E^{-1} \, \chi_E\]
where each $c_E$ is a certain non-zero element of $K$, called the
{\em Schur element} associated with $E$. (We refer to \cite[Chap.~7]{gepf} 
for basic facts about symmetric algebras.)

We shall further assume that there is a $K$-linear anti-involution
\[\dagger\colon \cH \rightarrow \cH, \qquad h \mapsto h^\dagger.\]
This allows us to define, for any finite-dimensional (left) $\cH$-module $M$,
a corresponding {\em contragredient} module $\hat{M}$. As a $K$-vector 
space, we have $\hat{M}= \Hom_K(M,K)$; the action of $h\in\cH$ on 
$f\in\hat{M}$ is determined by $(h.f)(m)=f(h^\dagger.m)$ for all $m \in M$. 

\begin{defn} \label{rem2} Let $M$ be a finite-dimensional (left) $\cH$-module.
We shall say that a bilinear map $(\;,\;) \colon M \times M \rightarrow K$ 
is $\cH$-invariant if 
\[(h.m,m')=(m,h^\dagger.m')\qquad\mbox{for all $h\in\cH$ and $m,m'\in M$}.\]
Via the isomorphism $\Hom_K(M,K) \otimes_K M \cong \Hom_K(M,M)$ 
(and an identification of $M$ with $\Hom_K(M,K)$ using dual bases), 
an $\cH$-invariant bilinear from on $M$ can also be interpreted as an
$\cH$-module homomorphism $\hat{M} \rightarrow M$, and vice versa. 
In particular, for $E \in \Irr(\cH)$, we have $E \cong \hat{E}$ if and 
only if there exists a non-degenerate $\cH$-invariant bilinear form on $E$; 
also note that a non-zero $\cH$-invariant bilinear form on $E$ is 
automatically non-degenerate (by Schur's Lemma). 
\end{defn}

Given any basis $B$ of $\cH$, we denote by 
$B^\vee=\{b^\vee \mid b \in B\}$ the corresponding dual basis, that is, we 
have
\[ \tau(b'b^\vee)=\left\{\begin{array}{cl} 1 & \quad \mbox{if $b=b'$},\\
0 & \quad \mbox{otherwise}.\end{array}\right.\]

\begin{defn} \label{def1} Let $B_0$ be a basis fo $\cH$. We say that
$B_0$ is {\em $\dagger$-symmetric} if $b^\dagger=b^\vee$ for all $b \in B_0$.
\end{defn}

The standard example is the case where $\cH=K[G]$ is the group algebra 
of a finite group $G$ over $K=\C$ and $\tau$ is the trace form defined 
by $\tau(1)=1$ and $\tau(g)=0$ for $g \in G$ such that $g \neq 1$. We 
have an anti-involution $\dagger\colon \cH \rightarrow \cH$ given by 
$g^\dagger=g^{-1}$; then $B_0=G$ is a $\dagger$-symmetric basis of 
$\cH$. Further examples are provided by the algebra $\tilde{J}$ in
Section~\ref{sec2} and by the ``based rings'' considered by Lusztig 
\cite{Lu4}.

\begin{rem} \label{rem1} Assume that there exists a $\dagger$-symmetric 
basis $B_0$. This implies that
\begin{equation*}
\tau(h)=\tau(h^\dagger) \qquad \mbox{for all $h \in \cH$}.\tag{a}
\end{equation*}
Indeed, write the identity element of $\cH$ as $1_{\cH}=\sum_{b\in 
B_0} \alpha_b\, b$ where $\alpha_b \in K$ for all $b \in B_0$. Then 
a straightforward computation shows that $\tau(b^\dagger)=\tau(1_{\cH}
b^\vee)=\alpha_b$ for all $b \in B_0$. Now, we certainly have $1_{\cH}=
1_{\cH}^\dagger=\sum_{b \in B_0} \alpha_b\, b^\vee$. Hence, similarly, 
we also obtain $\tau(b)=\tau(1_{\cH}^\dagger b)=\alpha_b$ for all $b\in 
B_0$. Thus, (a) holds. Now let $E \in \Irr(\cH)$. Then, clearly, we have 
\begin{equation*} 
\chi_{\hat{E}}(b)=\chi_E(b^\dagger)=\chi_E(b^\vee) \qquad \mbox{for 
all $b \in B_0$}.\tag{b}
\end{equation*}
This also implies that $c_E=c_{\hat{E}}$ since
\[ \tau(b)=\tau(b^\dagger)=\sum_{E \in \Irr(\cH)} c_E^{-1}\, 
\chi_E(b^\dagger)= \sum_{E \in \Irr(\cH)} c_E^{-1}\, \chi_{\hat{E}}(b) 
\qquad \mbox{for all $b \in B_0$}.\]
\end{rem}

At first sight, the condition in Definition~\ref{def1} looks rather
strong. But the following remark shows that $\dagger$-symmetric bases
of $\cH$ always exist under some quite natural assumptions. 

\begin{rem} \label{rem0} There exists a $\dagger$-symmetric basis of $\cH$
if the following two conditions are satisfied:
\begin{itemize}
\item[(a)] $\tau(h)=\tau(h^\dagger)$ for all $h \in \cH$.
\item[(b)] $K$ is sufficiently large (which means
here: $K$ contains sufficiently many square roots).
\end{itemize}
Indeed, consider the bilinear form $\cH \times \cH \rightarrow K$, 
$(h,h') \mapsto \tau(h'h^\dagger)$. By (a), this bilinear form is 
symmetric; furthermore, one easily sees that it is non-degenerate. Hence, 
since $\mbox{char}(K)=0$, there exists an orthogonal basis of $\cH$
with respect to that form. If now $K$ contains sufficiently many square
roots, then we can rescale the basis elements and obtain an orthonormal
basis of $\cH$; any such basis is $\dagger$-symmetric.
\end{rem}

We can now state the following two propositions which generalise classical
results concerning the Frobenius--Schur indicator for characters of finite
groups (see, for example, Etingof et al. \cite[\S 5.1]{eti11}) to 
symmetric algebras as above. 

\begin{prop} \label{prop21} Assume that $B_0$ is a $\dagger$-symmetric 
basis of $\cH$. Let $E \in \Irr(\cH)$ and define
\[\nu_E:=\frac{1}{c_E\dim E}\sum_{b\in B_0}\chi_E(b^2).\]
Then we have $\nu_E \in \{0,\pm 1\}$; furthermore, the following hold:
\begin{itemize}
\item[(a)] $\nu_E =0$ if and only if $E \not\cong \hat{E}$.
\item[(b)] $\nu_E =1$ if and only if $E \cong \hat{E}$ and there exists a 
non-degenerate, symmetric $\cH$-invariant bilinear form on $E$.
\item[(c)] $\nu_E =-1$ if and only if $E \cong \hat{E}$ and there exists a 
non-degenerate, alternating $\cH$-invariant bilinear form on $E$.
\end{itemize}
(In particular, $\nu_E$ does not depend on the choice of $B_0$.)
\end{prop}

\begin{proof} This very closely follows the original proof of Frobenius 
and Schur, as presented by Curtis \cite[Chap. IV, \S 3]{cur99}. We choose 
a basis of $E$ and obtain a corresponding matrix representation $\rho 
\colon \cH \rightarrow M_d(K)$ where $d=\dim E$. For $h \in \cH$ and 
$i,j \in \{1,\ldots,d\}$, we denote by $\rho_{ij}(h)$ the 
$(i,j)$-coefficient of $\rho(h)$. Taking the dual basis in $\hat{E}$, a 
matrix representation afforded by $\hat{E}$ is then given by 
$\hat{\rho}(b)=\rho(b^\vee)^\prime$ for all $b \in B_0$, where the prime 
denotes the transpose matrix. 

Assume first that $E \not\cong \hat{E}$. Then the Schur relations in 
\cite[7.2.2]{gepf} yield:
\[ \sum_{b \in B_0} \rho_{ij}(b)\,\hat{\rho}_{kl}(b^\vee)=0 \qquad
\mbox{ for all $i,j,k,l \in \{1,\ldots,d\}$}.\]
Using the above description of $\hat{\rho}$, we conclude that 
\[ \sum_{b \in B_0} \rho_{ij}(b)\,\rho_{lk}(b)=0 \qquad
\mbox{ for all $i,j,k,l \in \{1,\ldots,d\}$}.\]
Now let $l=j$ and $k=i$. Then summing over all $i,j$ yields 
\[ 0=\sum_{1\leq i,j \leq d} \sum_{b \in B_0} \rho_{ij}(b)\,\rho_{ji}(b)
=\sum_{b \in B_0} \sum_{1 \leq i \leq d} \rho_{ii}(b^2)=\sum_{b \in B_0}
\chi_E(b^2).\]
Thus, we have $\nu_E=0$ in this case, as required.

Now assume that $E\cong \hat{E}$. This means that there exists an
invertible matrix $P \in M_d(K)$ such that 
\[P\,\rho(b)=\rho(b^\vee)^\prime\,P\qquad \mbox{for all $b \in B_0$}.\]
A standard argument using Schur's Lemma (see \cite[p.~153]{cur99}) then 
shows that $P^\prime=\eta P$ where $\eta=\pm 1$. Note that a similar 
statement is true for any matrix $Q \in M_d(K)$ such that $Q\rho(b)= 
\rho(b^\vee)^\prime Q$ for all $b \in B_0$. Indeed, by Schur's Lemma, $Q$ 
will be a scalar multiple of $P$ and so $Q^\prime=\eta Q$, with the same 
$\eta$ as before. Now our given $P$ defines a bilinear form $(\;,\;) \colon 
E \times E \rightarrow K$; the fact that $P\rho(b)=\rho(b^\vee)^\prime P$ 
for all $b \in B_0$ means that $(\;,\;)$ is $\cH$-invariant. Thus, we have 
already shown that if $E \cong \hat{E}$, then there exists a non-degenerate 
$\cH$-invariant bilinear form on $E$ which is either symmetric or 
alternating. (Conversely, if such a  bilinear form exists, then $E \cong 
\hat{E}$; see Remark~\ref{rem1}.) It remains to see how $\eta$ is determined.

For this purpose, let $U \in M_d(K)$ be any matrix and define 
\[ Q_U:=\sum_{b \in B_0} \rho(b)^\prime\, U\,\rho(b)=\sum_{b\in B_0}
\hat{\rho}(b^\vee)\,U \,\rho(b).\] 
The second equality shows that $Q_U \rho(b)=\rho(b^\vee)^\prime Q_U$ for 
all $b \in B_0$; see \cite[7.1.10]{gepf}. Hence, as we just remarked, we must
have $Q_U^\prime=\eta Q_U$ and so 
\[ \sum_{1\leq i,j\leq d}\sum_{b\in B_0} \rho_{il}(b)\,u_{ij}\,\rho_{jk}(b)=
\eta\sum_{1\leq i,j\leq d}\sum_{b\in B_0}\rho_{ik}(b)\,u_{ij}\,\rho_{jl}(b)\] 
for all $k,l\in \{1,\ldots,d\}$, where we write $U=(u_{ij})$. Now take 
$U$ to be the matrix with coefficient $1$ at position $(k,l)$ and 
coefficient $0$, otherwise. Then we obtain
\[ \sum_{b \in B_0} \rho_{kl}(b)\,\rho_{lk}(b)=\eta\sum_{b \in B_0} 
\rho_{kk}(b)\,\rho_{ll}(b) \qquad\mbox{for fixed $k,l\in\{1,\ldots,d\}$}.\] 
Summing over all $k,l$ yields 
\[ \sum_{b \in B_0} \chi_E(b^2)=\eta\sum_{b \in B_0} \chi_E(b)^2.\]
Finally, since $E\cong\hat{E}$, we have $\chi_E(b)=\chi_E(b^\vee)$. Hence, 
the right hand side of the above identity equals $\eta\sum_{b \in B_0} 
\chi_E(b)\,\chi_E(b^\vee)$ which, by the orthogonality relations for the 
irreducible characters of $\cH$ (see \cite[7.2.4]{gepf}), equals $\eta\, 
c_E \dim E$. Thus, $\nu_E= \eta=\pm 1$, as required. 

Once the above statements are proved, it follows that for any $E \in
\Irr(\cH)$ we have $\nu_E\in \{0, \pm 1\}$ and the equivalences in 
(a), (b), (c) hold.
\end{proof}

In the standard example where $\cH=\C[G]$ for a finite group $G$, we have 
$c_E=|G|/\dim E$ for all $E \in \Irr(\cH)$ (see \cite[7.2.5]{gepf}). Hence, 
in this case, the formula for $\nu_E$ in Proposition~\ref{prop21} indeed 
is the classical formula for the Frobenius--Schur indicator. 

\begin{prop} \label{prop22} Assume that there exists a $\dagger$-symmetric 
basis $B_0$ of $\cH$. Then 
\[ \operatorname{trace}(\dagger \colon \cH \rightarrow \cH)
=\sum_{E \in \Irr(\cH)} \nu_E\, \dim E.\]
In particular, if $B_0=B_0^\vee$, then 
\[ |\{b \in B_0 \mid b^\vee=b\}|=\sum_{E \in \Irr(\cH)} \nu_E\, \dim E.\]
\end{prop}

\begin{proof}  The second equality certainly follows from the first:
under the given assumption on $B_0$, we have $\mbox{trace}(\dagger)=
|\{b \in B_0 \mid b^\vee=b\}|$.
In order to prove the first equality, we compute the trace of $\dagger$
with respect to a basis of $\cH$ arising from the Wedderburn decomposition. 
Let $E \in \Irr(\cH)$. Choosing a basis of $E$, we obtain a corresponding 
matrix representation $\rho \colon \cH \rightarrow M_d(K)$ where 
$d=\dim E$. We set 
\[ e_{ij}^E=\frac{1}{c_E}\sum_{b \in B_0} \rho_{ji}(b^\vee)\,b
\qquad \mbox{for $i,j \in \{1,\ldots,d\}$}.\]
Then, by \cite[7.2.7]{gepf}, the matrix $\rho(e_{ij}^E)$ has 
coefficient $1$ at position $(i,j)$ and coefficient $0$, otherwise; 
furthermore, $e_{ij}^E$ acts as zero on any simple $\cH$-module which
is not isomorphic to $E$. The elements 
\[ \{ e_{ij}^E \mid E \in \Irr(\cH), \; 1 \leq i,j \leq \dim E\}\]
form a $K$-basis of $\cH$. We shall now compute the trace of $\dagger$ 
with respect to this basis. First note that, since the dual basis of 
$B_0^\vee$ is $B_0$ and since $e_{ij}^E$ is independent of the choice of 
the basis of $\cH$ (see \cite[\S 7.2]{gepf}), we have 
\[ e_{ij}^E=\frac{1}{c_E}\sum_{b \in B_0}\rho_{ji}(b^\vee) b=
\frac{1}{c_E} \sum_{b \in B_0}\rho_{ji}(b)b^\vee\]
and so 
\[ (e_{ij}^E)^\dagger=\frac{1}{c_E} \sum_{b \in B_0} \rho_{ji}(b^\vee)
b=\frac{1}{c_E}\sum_{b \in B_0} \hat{\rho}_{ij}(b^\vee)b=
e_{ji}^{\hat{E}},\]
where we use the fact that $c_E=c_{\hat{E}}$; see Remark~\ref{rem1}.
This already shows that those $e_{ij}^E$ where $E \not\cong\hat{E}$
will not contribute to the trace of $\dagger$. So let us now assume that
$E \cong \hat{E}$. Let $d=\dim E$. Then there exists an invertible 
matrix $P\in M_d(K)$ such that $P\rho(b)=\rho(b^\vee)^\prime P$ for all 
$b \in B_0$. Write $P=(p_{ij})$ and $P^{-1}=(\tilde{p}_{ij})$. Then we have 
\[ (e_{ij}^E)^\dagger=\frac{1}{c_E}\sum_{b \in B_0}\rho_{ji}(b)\,b=
\frac{1}{c_E}\sum_{b \in B_0}\sum_{1\leq k,l\leq d} \tilde{p}_{jk}\,
p_{li}\, \rho_{lk}(b^\vee)\,b=\sum_{1\leq k,l\leq d} 
\tilde{p}_{jk}\,p_{li}\, e_{kl}^E.\]
The coefficient of $e_{ij}^E$ in the expression on the right hand side 
is $\tilde{p}_{ji}p_{ji}$. The contribution to the trace of $\dagger$ 
from basis vectors corresponding to $E$ will be the sum of all these 
terms. Now, we have $P^\prime=\nu_E P$; see the proof of 
Proposition~\ref{prop21}. Hence, the contribution from $E$ is 
\[ \sum_{1\leq i,j\leq d} \tilde{p}_{ji}p_{ji}=\nu_E \sum_{1\leq i,j
\leq d} \tilde{p}_{ij}p_{ji}=\nu_E\,\mbox{trace}(P^{-1}P)=\nu_E \dim E.\]
Consequently, we have $\mbox{trace}(\dagger)=\sum_E \nu_E \dim E$ where 
the sum runs over all $E \in \Irr(\cH)$ such that $E \cong \hat{E}$. 
Since $\nu_E=0$ for all $E\in \Irr(\cH)$ such that $E \not\cong\hat{E}$, 
this yields the desired formula.
\end{proof}


\begin{exmp}  \label{expreal} Let $B_0$ be a $\dagger$-symmetric basis
of $\cH$ and assume that $K \subseteq\R$. We claim that then $\nu_E=1$ 
for all $E \in \Irr(\cH)$. To see this, we adapt the classical argument 
for finite groups. Let $E \in \Irr(\cH)$. Choosing a basis of $E$, we 
obtain a corresponding matrix representation $\rho \colon \cH 
\rightarrow M_d(K)$ where $d=\dim E$. We set 
\[ Q:=\sum_{b \in B_0} \rho(b)^\prime\, \rho(b)=\sum_{b\in B_0}
\hat{\rho}(b^\vee)\,\rho(b).\] 
Clearly, $Q$ is symmetric. As in the proof of Proposition~\ref{prop21}, 
the second equality shows that $Q\rho(b)=\hat{\rho}(b)Q$ for all 
$b \in B_0$, so $Q$ defines a symmetric, $\cH$-invariant bilinear
form on $E$. Now, the diagonal coefficients of $Q$ are sums of squares 
of elements of $K$, at least some of which are non-zero (since $\rho(b)
\neq 0$ for at least some $b \in B_0$). Hence, since $K \subseteq \R$, 
these diagonal coefficients are non-zero and so $Q \neq 0$. By Schur's Lemma,
$Q$ is invertible. Thus, we are in case (b) of Proposition~\ref{prop21}. 
\end{exmp} 

Finally, we remark that there is an extensive literature on further 
generalisations of the Frobenius--Schur indicator, but usually this is done 
in the framework of Hopf algebras; see, for example, Guralnick--Montgomery 
\cite{gumo} and the references there. 

\section{The ring $\tilde{J}$} \label{sec2}
We shall now apply the results of the previous section to cells in finite 
Coxeter groups. Let $W$ be a finite Coxeter group and $S$ be a set of simple 
reflections in $W$. We fix a weight function $L \colon W \rightarrow \Z$ 
in the sense of Lusztig \cite{Lusztig03}, where we assume that $L(s) \geq 0$
for all $s \in S$. Using the Kazhdan--Lusztig basis in the generic 
Iwahori--Hecke algebra associated with $W,L$, we can define partitions of 
$W$ into left, right and two-sided cells. (Note that these notions depend
on $L$). 

The key tool to study these cells will be the theory around Lusztig's 
ring $J$, originally introduced in \cite{Lu2} in the equal parameter case. 
Subsequently, Lusztig \cite{Lusztig03} extended the theory to the general 
case, assuming that certain conjectural properties hold; see 
{\bf P1}--{\bf P15} in \cite[14.2]{Lusztig03}. In order to avoid the 
dependence on these conjectural properties, we shall work with a version
of Lusztig's ring introduced in \cite{myedin}. Let $\tilde{J}$ denote this
new version of $J$. The principal advantage of $\tilde{J}$ is that it can be 
constructed without any assumption on $W,L$. On the other hand, the results 
that are known about $\tilde{J}$ are not as strong as those for $J$ but, as 
we shall see, they are sufficient to deduce Theorem~\ref{propleft}. (See 
Remark~\ref{finrem} below for some comments on the relation between $J$ 
and $\tilde{J}$.)

We now recall the basic facts about the construction of $\tilde{J}$;
we use \cite[\S 1.5]{geja} as a reference. Let $K \subseteq \C$ be any 
field which is a splitting field for $W$. Let $\Irr_K(W)$ denote the set 
of simple $K[W]$-modules (up to isomorphism) and write 
\[ \Irr_K(W)=\{E^\lambda \mid \lambda \in \Lambda\} \qquad \mbox{(for 
some finite indexing set $\Lambda$)}.\]
For each $\lambda \in \Lambda$ let $M(\lambda)$ be a basis of $E^\lambda$. 
Then, by the construction in \cite[\S 1.4]{geja}, we obtain 
corresponding {\em leading matrix coefficients} 
\[ c_{w,\lambda}^{\fs\ft} \in K \qquad \mbox{where $w \in W$, $\lambda
\in \Lambda$ and $\fs,\ft \in M(\lambda)$}.\]
(The construction uses the generic Iwahori--Hecke algebra associated with 
$W,L$ and, hence, the above numbers depend on $L$.) For $x,y,z\in W$, we set 
\[ \tilde{\gamma}_{x,y,z}=\sum_{\lambda\in \Lambda} \sum_{\fs,\ft,\fu
\in M(\lambda)} f_\lambda^{-1} \,c_{x,\lambda}^{\fs\ft}\, c_{y,\lambda}^{\ft
\fu}\,c_{z,\lambda}^{\fu\fs},\]
where each $f_\lambda \in K$ is a non-zero element obtained from the 
corresponding Schur element of the generic Iwahori--Hecke algebras 
associated with $W,L$ (see \cite[1.3.1]{geja}). Now $\tilde{J}$ is an 
associative algebra over $K$, with a basis $\{t_w \mid w \in W\}$. The 
multiplication is given by 
\[ t_xt_y=\sum_{z \in W} \tilde{\gamma}_{x,y,z} t_{z^{-1}} \qquad \mbox{for
$x,y \in W$}.\]
There is an identity element given by $1_{\tilde{J}}=\sum_{w \in W}
\tilde{n}_wt_w$ where
\[\tilde{n}_w=\sum_{\lambda \in \Lambda} \sum_{\fs \in M(\lambda)}
f_\lambda^{-1}\, c_{w,\lambda}^{\fs\fs}\qquad\mbox{for all $w\in W$}.\] 
The algebra $\tilde{J}$ is symmetric with trace form $\tau \colon\tilde{J} 
\rightarrow K$, where  $\tau(t_w)=\tilde{n}_w$ for all $w \in W$. We also 
note that $\tilde{n}_w= \tilde{n}_{w^{-1}}$ and, hence, $\tau(t_w)=
\tau(t_{w^{-1}})$ for all $w \in W$ (see  \cite[1.5.3(c)]{geja}). 
Furthermore, the map
\[ \dagger \colon \tilde{J} \rightarrow \tilde{J}, \qquad t_w \mapsto 
t_{w^{-1}},\]
is an anti-involution of $\tilde{J}$ and $B_0=\{t_w \mid w \in W\}$ is a 
$\dagger$-symmetric basis of $\tilde{J}$. Finally, $\tilde{J}$ 
is split semisimple and we have a corresponding labelling 
\[ \Irr(\tilde{J})=\{\tilde{E}^\lambda \mid \lambda \in \Lambda\} \quad
\mbox{such that} \quad \dim E^\lambda=\dim \tilde{E}^\lambda \quad
\mbox{for all $\lambda\in \Lambda$}.\] 
We have $\tau=\sum_{\lambda \in \Lambda} f_\lambda^{-1} \chi_{\tilde{E}}$, 
hence the numbers $f_\lambda$ ($\lambda \in \Lambda$) are the Schur 
elements of $\tilde{J}$. (For all these facts, see \cite[\S 1.5]{geja}.) 

Now, by imitating the original definitions of Kazhdan and Lusztig 
\cite{KaLu}, one can define partitions of $W$ into left, right and two-sided 
``$\tilde{J}$-cells''; see \cite[\S 1.6]{geja}. (If we just say ``left 
cell'', ``right cell'' or ``two-sided cell'', then this is always meant to 
be a cell in the sense of Kazhdan and Lusztig, with respect to the given 
weight function~$L$.) Here are some of the essential properties that we 
shall need:
\begin{itemize}
\item[{\bf (J1)}] If $\tilde{\gamma}_{x,y,z}\neq 0$, then $x,y^{-1}$ 
belong to the same left $\tilde{J}$-cell, $y,z^{-1}$ belong to the same left 
$\tilde{J}$-cell and $z,x^{-1}$ belong to the same left $\tilde{J}$-cell.
(See \cite[1.6.4]{geja}.)
\item[{\bf (J2)}] For $\lambda \in \Lambda$, the set of all $w \in W$ 
such that $c_{w,\lambda}^{\fs\ft}\neq 0$ for some $\fs,\ft \in M(\lambda)$ 
is contained in a two-sided $\tilde{J}$-cell. (See \cite[1.6.11]{geja}.)
\item[{\bf (J3)}] If $\Gamma$ is a left cell of $W$, then $\Gamma$ is a 
union of left $\tilde{J}$-cells. A similar statement holds for right cells 
and two-sided cells. (See \cite[2.1.20]{geja}.)
\end{itemize}
Now let $\fC$ be a left $\tilde{J}$-cell or, slightly more generally, a 
union of left $\tilde{J}$-cells. Then {\bf (J1)} implies that 
\[ [\fC]_{\tilde{J}}=\langle t_x\mid x\in\fC\rangle_K\subseteq\tilde{J}\]
is a left ideal in $\tilde{J}$ and, thus, can be viewed as a left 
$\tilde{J}$-module. For $\lambda \in \Lambda$, denote by $\tilde{m}(\fC,
\lambda)$ the multiciplity of $\tilde{E}^\lambda \in \Irr(\tilde{J})$ as 
an irreducible constituent of $[\fC]_{\tilde{J}}$. Clearly, if $\fC_1,
\ldots,\fC_r$ are the left $\tilde{J}$-cells contained in $\fC$, then 
\[ [\fC]_{\tilde{J}}=\bigoplus_{1\leq i \leq r} [\fC_i]_{\tilde{J}} \quad 
\mbox{and}\quad \tilde{m}(\fC,\lambda)=\sum_{1\leq i \leq r} 
\tilde{m}(\fC_i, \lambda) \quad \mbox{for all $\lambda \in \Lambda$}.\]

\begin{lem} \label{copylus} Let $\fC,\fC'$ be left
$\tilde{J}$-cells of $W$.
\begin{itemize}
\item[(a)] We have $\Hom_{\tilde{J}}([\fC]_{\tilde{J}},[\fC']_{\tilde{J}})
=\{0\}$ unless $\fC,\fC'$ are contained in the same two-sided 
$\tilde{J}$-cell.
\item[(b)] In general, we have $\dim \Hom_{\tilde{J}}([\fC]_{\tilde{J}}, 
[\fC']_{\tilde{J}})=|\fC' \cap (\fC')^{-1}|$; in particular, $\fC\cap 
(\fC')^{-1}=\varnothing$ unless $\fC,\fC'$ are contained in the same 
two-sided $\tilde{J}$-cell.
\item[(c)] If $\fC=\fC'$, then the subspace $\tilde{J}_\fC:=\langle t_w 
\mid w \in \fC \cap  \fC^{-1} \rangle_K \subseteq \tilde{J}$ is a 
subalgebra isomorphic to $\End_{\tilde{J}}([\fC]_{\tilde{J}})$. Furthermore,
$\tilde{J}_{\fC}$ is split semisimple and has identity element
$1_{\fC}:=\sum_{w \in \fC \cap \fC^{-1}} \tilde{n}_w\,t_w$.
\end{itemize}
\end{lem}

\begin{proof} (a) Assume that $\Hom_{\tilde{J}}([\fC]_{\tilde{J}},
[\fC']_{\tilde{J}})\neq\{0\}$. This means that there is some $\lambda \in 
\Lambda$ such that $\tilde{m}(\fC,\lambda)>0$ and $\tilde{m}(\fC',\lambda)
>0$. By \cite[1.8.1]{geja}, there exist $w \in \fC$ and $w' \in \fC'$ such 
that $c_{w,\lambda}^{\fs\fs}\neq 0$ and $c_{w',\lambda}^{\ft\ft}\neq 0$ 
for some $\fs,\ft \in M(\lambda)$. By {\bf (J2)}, $w$ and $w'$ are 
contained in the same two-sided $\tilde{J}$-cell. Consequently, $\fC$ 
and $\fC'$ must be contained in the same two-sided $\tilde{J}$-cell.

(b) This is modelled on the argument of Lusztig \cite[12.15]{LuBook}. 
First we show that
\begin{equation*}
\dim \Hom_{\tilde{J}}([\fC]_{\tilde{J}}, [\fC']_{\tilde{J}})\geq 
|\fC \cap (\fC')^{-1}|.\tag{$*$}
\end{equation*}
If $\fC\cap (\fC')^{-1}=\varnothing$, this is clear. Now assume that 
$\fC\cap (\fC')^{-1}\neq \varnothing$. Let $y \in \fC\cap (\fC')^{-1}$ and 
$x \in \fC$. Then $t_x t_{y^{-1}}=\sum_{z \in W} \tilde{\gamma}_{x,y^{-1},z} 
t_{z^{-1}}$. If $\tilde{\gamma}_{x,y^{-1},z} \neq 0$, then $y^{-1},z^{-1}$ 
belong to the same left $\tilde{J}$-cell and so $z^{-1} \in \fC'$; see 
{\bf (J1)}. It follows that we have a well-defined left $\tilde{J}$-module 
homomorphism
\[ \varphi_y \colon [\fC]_{\tilde{J}} \rightarrow [\fC']_{\tilde{J}}, 
\qquad t_x \mapsto t_xt_{y^{-1}} \;(x \in \fC).\]
We claim that the collection of maps $\{\varphi_y \mid y \in \fC\cap 
(\fC')^{-1}\}$ is linearly independent in $\Hom_K([\fC]_{\tilde{J}},
[\fC']_{\tilde{J}})$. Indeed, assume that 
\[ \sum_{y \in \fC\cap (\fC')^{-1}} \alpha_y \,\varphi_y=0 \qquad 
\mbox{where $\alpha_y \in K$ for all $y \in \fC\cap (\fC')^{-1}$}.\]
Let $x \in \fC \cap (\fC')^{-1}$. Applying the above linear 
combination to $t_x$ and then evaluating the trace form $\tau$ on the 
resulting expression, we obtain
\[ 0=\sum_{y \in \fC\cap (\fC')^{-1}} \alpha_y \tau(\varphi_y(t_x))=
\sum_{y \in \fC\cap (\fC')^{-1}} \alpha_y \tau(t_xt_{y^{-1}})=
\alpha_x,\]
using the fact that $B_0=\{t_w \mid w \in W\}$ is a $\dagger$-symmetric 
basis of $\tilde{J}$. Thus, we have $\alpha_x=0$ for all $x\in \fC\cap 
(\fC')^{-1}$, as required. This certainly implies that ($*$) holds.
We can then complete the proof by a counting argument, exactly as in
\cite[12.15]{LuBook}. In particular, this shows that 
\[ \{\varphi_y \mid y \in \fC \cap (\fC')^{-1}\} \mbox{ is a vector 
space basis of $\Hom_{\tilde{J}}([\fC]_{\tilde{J}},[\fC']_{\tilde{J}})$}.\]

(c) Let $\fC=\fC'$. By (b), we have $\fC\cap \fC^{-1}\neq \varnothing$. 
Let $x,y \in \fC\cap \fC^{-1}$ and write $t_xt_y=\sum_{z \in W} 
\tilde{\gamma}_{x,y,z}t_{z^{-1}}$. If $\tilde{\gamma}_{x,y,z}\neq 0$ 
then $y,z^{-1}$ belong to the same left $\tilde{J}$-cell and $z,x^{-1}$ 
belong to the same left $\tilde{J}$-cell; see again {\bf (J1)}. Thus, we 
must have $z \in \fC \cap \fC^{-1}$. This shows that $\tilde{J}_\fC$ is a 
subalgebra of $\tilde{J}$. Using now the construction in the proof of 
(b), we obtain an isomorphism of vector spaces
\[\varphi\colon\tilde{J}_\fC\rightarrow\End_{\tilde{J}}([\fC]_{\tilde{J}}),
\qquad t_y \mapsto \varphi_y \quad (y \in \fC \cap \fC^{-1}).\]
We note that, for any $h \in \tilde{J}_{\fC}$, the map $\varphi(h)$ is 
given by right multiplication with $h^\dagger$. This certainly implies that
$\varphi$ is an algebra homomorphism. 

Finally, being isomorphic to the endomorphism algebra of a module of 
a split semisimple algebra, $\tilde{J}_\fC$ itself has an identity 
element and is split semisimple. Let $1_{\fC}$ be the identity
element and write $1_{\fC}= \sum_{w \in \fC\cap \fC^{-1}} \alpha_w\, 
t_w$ where $\alpha_w \in K$. If $x \in \fC\cap \fC^{-1}$, then $t_x 
\in \tilde{J}_{\fC}$ and so 
\[\tilde{n}_x=\tau(t_{x^{-1}})=\tau(t_{x^{-1}}1_{\fC})=\sum_{w \in 
\fC\cap \fC^{-1}} \alpha_w\, \tau(t_{x^{-1}}t_w)=\alpha_x.\]
Thus, $1_{\fC}$ has the required expression.
\end{proof}

\begin{rem} \label{copylus1} Let $\fC$ be a left $\tilde{J}$-cell. Recall
that, by definition, the left $\tilde{J}$-module $[\fC]_{\tilde{J}}=
\langle t_w \mid w \in \fC\rangle_k$ is a left ideal in $\tilde{J}$. By 
Lemma~\ref{copylus}(c), the element $1_{\fC}$ is an idempotent in 
$\tilde{J}$, and it is contained in $[\fC]_{\tilde{J}}$. 
In fact, we claim that 
\[ [\fC]_{\tilde{J}}=\tilde{J}1_{\fC}.\]
Indeed, since $1_{\fC} \in [\fC]_{\tilde{J}}$, it is clear that
$\tilde{J}1_{\fC} \subseteq [\fC]_{\tilde{J}}$. Conversely, we note that 
right multiplication by $1_{\tilde{J}}$ is the identity element of 
$\End_{\tilde{J}}([\fC]_{\tilde{J}})$; see the proof of 
Lemma~\ref{copylus}(c). Thus, for any $w \in \fC$, we have $t_w=t_w
1_{\fC} \in \tilde{J}1_{\fC}$, as required. 
\end{rem}

For any subset $X \subseteq W$, we denote by $X_{(2)}$ the set of 
involutions in $X$.

\begin{lem} \label{lem31} Let $\fC$ be a union of left $\tilde{J}$-cells of
$W$. Then we have
\[ |\fC_{(2)}|=\sum_{\lambda \in \Lambda} \tilde{m}(\fC,\lambda).\]
\end{lem}

\begin{proof} Since $\R$ is known to be a splitting field for $W$ (see 
\cite[6.3.8]{gepf}), we will assume in this proof that $K \subseteq \R$. 
Let $\fC_1,\ldots,\fC_r$ be the left $\tilde{J}$-cells which are contained 
in $\fC$. Then, clearly, $\fC_{(2)}$ is the union of the sets of involutions 
in $\fC_1,\ldots,\fC_r$; furthermore, as already noted above, we have 
$\tilde{m}(\fC, \lambda)=\tilde{m}(\fC_1,\lambda)+ \ldots+\tilde{m}(\fC_r,
\lambda)$ for all $\lambda \in \Lambda$. Thus, it will be sufficient to 
deal with the case where $r=1$ and $\fC=\fC_1$ is just one left 
$\tilde{J}$-cell. In this case, consider the split semisimple algebra 
$\cH:=\tilde{J}_{\fC}$; see Lemma~\ref{copylus}(c). We note that $\dagger$ 
restricts to an anti-involution of $\cH$ which we denote by the same symbol. 
Furthermore, $\tau$ restricts to a trace form on $\cH$ where $B_{0,\fC}=
\{t_w \mid w \in \fC \cap \fC^{-1}\}$ is a $\dagger$-symmetric basis of 
$\tilde{J}_{\fC}$ such that $B_{0,\fC}=B_{0,\fC}^\vee$. Thus, we can apply 
the results in Section~\ref{sec1} to $\cH$. Since $K \subseteq \R$ and 
since $\fC_{(2)} \subseteq \fC \cap \fC^{-1}$, we conclude that 
\[ |\fC_{(2)}|=\sum_{M\in \Irr(\cH)} \dim M \qquad \mbox{(see 
Proposition~\ref{prop22} and Example~\ref{expreal})}.\]
It remains to note that, since $\tilde{J}$ is split semisimple and since 
we have an isomorphism $\cH\cong \End_{\tilde{J}}([\fC]_{\tilde{J}})$, there 
is a bijection between $\Irr(\cH)$ and the set of simple $\tilde{J}$-modules 
which appear as constituents of $[\fC]_{\tilde{J}}$; furthermore, if 
$\tilde{E}^\lambda \in \Irr(\tilde{J})$ is such a constituent, then the 
corresponding simple $\cH$-module has dimension $\tilde{m}(\fC,\lambda)$. 
(This follows from simple facts about Hom functors; see, e.g., 
\cite[4.1.3]{geja}.)
\end{proof}

\begin{cor} \label{cormult} Let $\fC$ be a left $\tilde{J}$-cell.
Then $[\fC]_{\tilde{J}}$ is multiplicity-free if and only if $\fC_{(2)}=
\fC \cap \fC^{-1}$.
\end{cor}

\begin{proof} Recall that $\fC_{(2)}\subseteq \fC\cap\fC^{-1}$. By 
Lemma~\ref{lem31}, we have $|\fC_{(2)}|= \sum_{\lambda \in \Lambda} 
\tilde{m}(\fC,\lambda)$. On the other hand, by Lemma~\ref{copylus}(b), 
we have $|\fC\cap \fC^{-1}|=\sum_{\lambda \in \Lambda} \tilde{m}(\fC,
\lambda)^2$. Hence, if $[\fC]_{\tilde{J}}$ is multiplicitly-free, then 
$\tilde{m}(\fC,\lambda)=\tilde{m}(\fC,\lambda)^2$ for all $\lambda \in 
\Lambda$ and so $\fC_{(2)}=\fC\cap \fC^{-1}$. On the other hand, if 
$\fC_{(2)}=\fC\cap \fC^{-1}$, then $\sum_{\lambda \in \Lambda} \tilde{m}
(\fC,\lambda)=\sum_{\lambda \in \Lambda} \tilde{m}(\fC,\lambda)^2$ and 
so $\tilde{m}(\fC,\lambda) \in \{0,1\}$ for all $\lambda\in \Lambda$.
\end{proof}

\begin{rem} \label{ll4} Let $\lambda \in \Lambda$ and $w \in W$. By 
\cite[1.5.7]{geja}, we have 
\begin{equation*}
c_{w,\lambda}=\mbox{trace}(t_w,\tilde{E}^\lambda) \qquad \mbox{where} \qquad 
c_{w,\lambda}:=\sum_{\fs \in M(\lambda)} c_{w,\lambda}^{\fs\fs}. \tag{a}
\end{equation*}
Thus, up to signs, the numbers $c_{w,\lambda}$ are the leading coefficients
of character values as defined by Lusztig \cite{Lu4}. We claim that
\begin{equation*}
c_{w,\lambda}=0 \quad \mbox{unless $w,w^{-1}$ belong to the same left
$\tilde{J}$-cell}. \tag{b}
\end{equation*}
Indeed, assume that $c_{w,\lambda} \neq 0$. Using (a), we conclude that 
$t_w$ can not be nilpotent. Consequently, $t_w^2\neq 0$ and so 
$\tilde{\gamma}_{w,w,x} \neq 0$ for some $x \in W$. Hence, by {\bf (J1)}, 
the elements $w,w^{-1}$ must belong to the same left $\tilde{J}$-cell, as 
claimed. (In the equal parameter case, this argument is due to Lusztig 
\cite[3.5]{Lu4}.)
\end{rem}

\begin{exmp} \label{remc} Recall that $1_{\tilde{J}}=\sum_{w \in W} 
\tilde{n}_w t_w$. Let $\cD=\{w \in W \mid \tilde{n}_w\neq 0\}$. If 
{\bf P1}--{\bf P15} in \cite[14.2]{Lusztig03} were known to hold for 
$W,L$, then we could deduce that every element of $\cD$ is an involution. 
In the present context, we can at least show that $w,w^{-1}$ belong to 
the same left $\tilde{J}$-cell. Indeed, if $\tilde{n}_w\neq 0$, then the 
defining equation shows that $c_{w,\lambda} \neq 0$ for some $\lambda \in 
M(\lambda)$. So Remark~\ref{ll4}(b) implies that $w,w^{-1}$ belong to the 
same left $\tilde{J}$-cell, as claimed. In particular, if we are in a case 
where all $\tilde{J}$-modules $[\fC]_{\tilde{J}}$ are multiplicity-free (for
any left $\tilde{J}$-cell $\fC$ of $W$), then $w^2=1$ for all $w \in \cD$ 
(see Corollary~\ref{cormult}).
\end{exmp}

Now let $\Gamma$ be a left cell of $W$. Recall that we have a corresponding
left $K[W]$-module $[\Gamma]_1$. For any $\lambda \in \Lambda$, let 
$m(\Gamma,\lambda)$ denote the multiplicity of $E^\lambda \in \Irr_K(W)$ as 
an irreducible constituent of $[\fC]_1$. Now, $\Gamma$ is a union of left 
$\tilde{J}$-cells; see {\bf (J3)}. Thus, in order to complete the proof of 
Theorem~\ref{propleft}, we need to compare the multiplicities $m(\Gamma,
\lambda)$ and $\tilde{m}(\Gamma,\lambda)$. 

\begin{lem} \label{lem32} With the above notation, we have $\tilde{m}
(\Gamma,\lambda)=m(\Gamma,\lambda)$ for any $\lambda \in \Lambda$.
Consequently, Theorem~\ref{propleft} holds. 
\end{lem}

\begin{proof}  Let $\lambda \in \Lambda$. Using \cite[Prop.~4.7]{my02} (see 
also the proof of \cite[2.2.4]{geja}), we have 
\[ \sum_{\fs,\ft \in M(\lambda)} \sum_{w \in \Gamma} c_{w,\lambda}^{\fs\ft}
\, c_{w^{-1},\lambda}^{\ft\fs}=m(\Gamma,\lambda)\,f_\lambda\,|M(\lambda)|.\]
On the other hand, let $\fC_1,\ldots,\fC_r$ be the left $\tilde{J}$-cells
which are contained in $\Gamma$. Then, using \cite[1.8.1]{geja}, we have
\[ \sum_{\fs,\ft \in M(\lambda)} \sum_{w \in \fC_i} c_{w,\lambda}^{\fs\ft}
\, c_{w^{-1},\lambda}^{\ft\fs}=\tilde{m}(\fC_i,\lambda)\,f_\lambda \,
|M(\lambda)| \qquad \mbox{for $i=1,\ldots,r$}.\]
Summing these identities over $i=1,\ldots,r$ and using the fact that
$\tilde{m}(\Gamma,\lambda)=\tilde{m}(\fC_1,\lambda)+\ldots+\tilde{m}
(\fC_r,\lambda)$, we obtain 
\[ \sum_{\fs,\ft \in M(\lambda)} \sum_{w \in \Gamma} c_{w,\lambda}^{\fs\ft}
\, c_{w^{-1},\lambda}^{\ft\fs}=\tilde{m}(\Gamma,\lambda)\,f_\lambda\,
|M(\lambda)|.\]
We conclude that $m(\Gamma,\lambda)=\tilde{m}(\Gamma,\lambda)$, as required.
In combination with Lemma~\ref{lem31}, this yields that $|\Gamma_{(2)}|=
\sum_{\lambda \in \Lambda} m(\Gamma,\lambda)$. Note that the right hand 
side is just the number of terms in a decomposition of $[\Gamma]_1$ as a 
direct sum of simple $K[W]$-modules. Thus, Theorem~\ref{propleft} is proved.
\end{proof}

\begin{cor}[See Lusztig \protect{\cite[5.8]{LuBook}} in the equal 
parameter case] \label{cor32} Let $\Gamma$ be a left cell of $W$ and 
$\lambda,\mu \in \Lambda$. Then
\[ \sum_{w \in \Gamma} c_{w,\lambda}\, c_{w^{-1},\mu}=\left\{
\begin{array}{cl} m(\Gamma,\lambda)\,f_\lambda & \qquad 
\mbox{if $\lambda=\mu$},\\ 0 & \qquad \mbox{otherwise}.\end{array}\right.\]
\end{cor}

\begin{proof} Let $\fC_1,\ldots, \fC_r$ be the left $\tilde{J}$-cells which 
are contained in $\Gamma$. By \cite[1.8.1]{geja}, it is already known that,
for any $i \in \{1,\ldots,r\}$, we have 
\[ \sum_{w \in \fC_i} c_{w,\lambda}\, c_{w^{-1},\mu}=\left\{
\begin{array}{cl} \tilde{m}(\fC_i,\lambda)\,f_\lambda & \qquad 
\mbox{if $\lambda=\mu$},\\ 0 & \qquad \mbox{otherwise}.\end{array}\right.\]
We sum these identities over all $i=1,\ldots,r$. Then it remains to use 
Lemma~\ref{lem32} and the fact that $\tilde{m}(\Gamma,\lambda)=
\tilde{m}(\fC_1,\lambda)+ \ldots+\tilde{m}(\fC_r,\lambda)$.
\end{proof}

\begin{cor}[See Lusztig \protect{\cite[12.15]{LuBook}} in the equal 
parameter case] \label{cor30} Let $\Gamma,\Gamma'$ be left cells of $W$. 
Then 
\[ \dim \Hom_{K[W]}([\Gamma]_1,[\Gamma']_1)=|\Gamma\cap (\Gamma')^{-1}|.\]
Furthermore, we have $\Hom_{K[W]}([\Gamma]_1,[\Gamma']_1)=\{0\}$ unless 
$\Gamma, \Gamma'$ are contained in the same two-sided cell of $W$.
\end{cor}

\begin{proof} We have 
\begin{align*}
\dim \Hom_{K[W]}([\Gamma]_1,[\Gamma']_1)&= \sum_{\lambda \in \Lambda} 
m(\Gamma,\lambda)\, m(\Gamma',\lambda)\\
&=\sum_{\lambda \in \Lambda} \tilde{m}(\Gamma,\lambda)\, 
\tilde{m}(\Gamma',\lambda) \qquad \mbox{(see Lemma~\ref{lem31})}\\
&=\dim \Hom_{\tilde{J}}([\Gamma]_{\tilde{J}},[\Gamma']_{\tilde{J}}).
\end{align*}
Thus, it is sufficient to show that $\dim \Hom_{\tilde{J}}
([\Gamma]_{\tilde{J}},[\Gamma']_{\tilde{J}})=|\Gamma\cap (\Gamma')^{-1}|$. 
To see this, let $\fC_1,\ldots, \fC_r$ be the left $\tilde{J}$-cells which 
are contained in $\Gamma$ and let $\fC_1',\ldots,\fC_s'$ be the left 
$\tilde{J}$-cells which are contained in $\Gamma'$. Then we have 
\[ \dim\Hom_{\tilde{J}}([\Gamma]_{\tilde{J}}, [\Gamma']_{\tilde{J}})=
\sum_{1\leq i\leq r} \sum_{1\leq j \leq s} \dim\Hom_{\tilde{J}}
([\fC_i]_{\tilde{J}}, [\fC_j']_{\tilde{J}}),\]
and so the desired equality immediately follows from Lemma~\ref{copylus}(b). 
Finally, if $\Gamma,\Gamma'$ are not contained in the same two-sided cell, 
then $\fC_i,\fC_j$ (for any $i,j$) are not contained in the same 
two-sided $\tilde{J}$-cell; see {\bf (J3)}. Thus, Lemma~\ref{copylus}(a) 
and the above formula show that $\Hom_{K[W]}([\Gamma]_1,[\Gamma']_1)=
\{0\}$ in this case.
\end{proof}

\begin{exmp} \label{expl} Let $\Gamma,\Gamma'$ be left cells of $W$ which 
are contained in the same two-sided cell. If $L(s)=1$ for all $s \in S$, it 
is known that we always have $\Gamma \cap (\Gamma')^{-1}\neq \varnothing$; 
see Lusztig \cite[12.16]{LuBook}.~--~For general $L$, it can happen that
$\Gamma\cap (\Gamma')^{-1}=\varnothing$; see \cite[Cor.~4.8]{mykl} (case 
``$b=2a$'' in Table~2) for an example in type $F_4$. 
\end{exmp}

\begin{cor} \label{cor31} Let $\Gamma$ be a left cell of $W$. Then 
$[\Gamma]_1$ is multiplicity-free if and only if $\Gamma_{(2)}=\Gamma 
\cap \Gamma^{-1}$.
\end{cor}

\begin{proof} Once Theorem~\ref{propleft} and Corollary~\ref{cor30} are
established, this follows by an argument entirely analogous to that in
Corollary~\ref{cormult}.
\end{proof}

Now let $\ulc$ be a two-sided cell of $W$. Given $E^\lambda\in \Irr_K(W)$, 
we write $E^\lambda\sim_L \ulc$ and say that $E^\lambda$ belongs to $\ulc$ 
if $E^\lambda$ is a constituent of some $[\Gamma]_1$ where $\Gamma$ is a 
left cell which is contained in $\ulc$.  

\begin{cor} \label{proptwo} The number of involutions in a two-sided cell
$\ulc$ of $W$ is equal to $\sum_\lambda \dim E^\lambda$ where the sum runs 
over all $\lambda \in \Lambda$ such that $E^\lambda \sim_L \ulc$.
\end{cor}

\begin{proof} Let $\Gamma^1, \ldots,\Gamma^m$ be all the left cells of $W$ 
(with respect to the given $L$). Then the direct sum $\bigoplus_{1 \leq i 
\leq m} [\Gamma^i]_1$ is isomorphic to the regular representation of $W$ 
and so each $E^\lambda\in \Irr_K(W)$ appears with multiplicity equal to 
$\dim E^\lambda$ in that direct sum. Now, by Corollary~\ref{cor30}, we 
have $\Hom_{K[W]} ([\Gamma^i]_1, [\Gamma^j]_1)=\{0\}$ whenever $\Gamma^i$, 
$\Gamma^j$ are not contained in the same two-sided cell of $W$. Consequently,
if $I$ denotes the set of all $i \in \{1,\ldots,m\}$ such that $\Gamma^i 
\subseteq \ulc$, then we have 
\[\sum_{i \in I} [\Gamma^i]_1=\sum_{\lambda \in \Lambda \,:\, E^\lambda 
\sim_L \ulc} (\dim E^\lambda) \,E^\lambda\]
(in the appropriate Grothendieck group of representations). Thus, the
number of terms in a decomposition of $\bigoplus_{i \in I} [\Gamma^i]_1$ 
as a direct sum of irreducible representations is equal to $\sum_\lambda
\dim E^\lambda$ where the sum runs over all $\lambda \in \Lambda$ such 
that $E^\lambda \sim_L \ulc$. On the other hand, by Theorem~\ref{propleft}, 
this number is also equal to the sum $\sum_{i \in I} |\Gamma_{(2)}^i|$ 
which is just the number of involutions in~$\ulc$.
\end{proof}

\begin{exmp}[Lusztig] \label{exp1} Assume that we are in the equal parameter 
case where $L(s)=1$ for all $s \in S$. Let $\ulc$ be a two-sided cell of $W$. 
By \cite[Chap.~4]{LuBook}, one can attach a certain finite group $\cG=
\cG_{\ulc}$ to $\ulc$ (or the corresponding family of $\Irr_K(W)$). Assume 
now that $W$ is of classical type. Then $|\cG_{\ulc}|=2^d$ for some $d 
\geq 0$ and, by \cite[12.17]{LuBook}, it is known that $[\Gamma]_1$ is 
multiplicity-free with exactly $2^d$ irreducible constituents, for every
left cell $\Gamma \subseteq \ulc$. Hence, $\Gamma_{(2)}=\Gamma\cap 
\Gamma^{-1}$ also has cardinality $2^d$ for any left cell $\Gamma\subseteq 
\ulc$. Now let $E_\ulc$ be the unique special representation which belongs
to $\ulc$ (see \cite[4.14, 5.25]{LuBook}). Since $E_\ulc$ occurs with 
multiplicity $1$ in $[\Gamma]_1$ for every left cell $\Gamma\subseteq\ulc$,
we conclude that $|\ulc_{(2)}|= 2^d \dim E_\ulc$. Combining this with 
Corollary~\ref{proptwo}, we obtain the following identity:
\[ 2^d \dim E_\ulc=\sum_{\lambda \in \Lambda\,:\, E^\lambda \sim_L \ulc} 
\dim E^\lambda,\]
which shows that the order of the group $\cG_{\ulc}$ is 
determined by the set of all $E^\lambda\in \Irr_K(W)$ which belong to 
$\ulc$. (If $W$ is of exceptional type, then such an identity 
will not hold in general.) 
\end{exmp}

\begin{rem} \label{finrem} Assume that the conjectural properties 
{\bf P1}--{\bf P15} in \cite[14.2]{Lusztig03} hold for $W,L$. Then we do
have $J=\tilde{J}$; see \cite[2.3.16]{geja}. In particular, this implies
that:
\begin{itemize}
\item $\tilde{\gamma}_{x,y,z} \in \Z$ and $\tilde{n}_w=\pm 1$ for all
$x,y,z,w\in W$;
\item every left $\tilde{J}$-cell contains a unique element of $\cD=
\{w \in W \mid \tilde{n}_w\neq 0\}$;
\item the left cells of $W$ are precisely the left $\tilde{J}$-cells.
\end{itemize}
(See \cite[2.5.3]{geja}.) It would be highly interesting to prove these 
statements directly, without reference to {\bf P1}--{\bf P15}; at present, 
we do not see any way of doing this. In \cite{mypamq}, we have formulated 
a somewhat different set of conjectural properties which, in some cases, are 
easier to verify than {\bf P1}--{\bf P15}. However, the case where $W$ if 
of type $B_n$ and $L$ is a general weight function remains completely open.
\end{rem}

\medskip
\noindent {\bf Acknowledgements.} I wish to thank George Lusztig for
a discussion at the Isaac Newton Institute (Cambridge, September 2011)
where he suggested that Corollary~\ref{proptwo} might be true in general.

\end{document}